\documentclass{amsart}

\usepackage{amssymb}
\usepackage[all]{xy}

\theoremstyle{plain}
\newtheorem{theorem}{Theorem}
\newtheorem{lemma}[theorem]{Lemma}
\newtheorem{proposition}[theorem]{Proposition}
\newtheorem{claim}[theorem]{Claim}

\theoremstyle{definition}
\newtheorem{definition}[theorem]{Definition}
\newtheorem{remark}[theorem]{Remark}
\newtheorem{example}[theorem]{Example}

\newcommand{\RR}{{\mathbb{R}}}

\begin{document}

\title[Biased infinity Laplacian Boundary Problem]{Biased infinity Laplacian Boundary Problem \\ on finite graphs}

\author{Yuval Peres}
\email{yuval@yuvalperes.com}

\author{Zoran {\v{S}}uni{\'c}}
\address{Deptartment of Mathematics, Hofstra University, Hempstead, NY 11549, USA}
\email{zoran.sunic@hofstra.edu}

\begin{abstract}
We provide an algorithm, running in polynomial time in the number of vertices, computing the unique solution to the biased infinity Laplacian Boundary Problem on finite graphs. The algorithm is based on the general outline and approach taken in the corresponding algorithm for the unbiased case provided by Lazarus et al. The new ingredient is an adjusted (biased) notion of a slope of a function on a path in a graph. The algorithm can be used to determine efficiently numerical approximations to the viscosity solutions of biased infinity Laplacian PDEs.
\end{abstract}



\maketitle


\section{Biased infinity Laplacian and tug-of-war on graphs}

Connections between the infinity Laplacian PDE, biased or unbiased, and tug-of-war games on graphs were explored in~\cite{peres-s-s-w:tug, peres-p-s:biased} to show that the corresponding Boundary Problem has a unique viscosity solution under broad conditions. As observed in~\cite{peres-p-s:biased}, numerical approximations to such solutions in the unbiased case can be efficiently obtained from the algorithm by Lazarus et al.~\cite{lazarus-etal:algorithm} computing the infinity harmonic extension on a finite graph in polynomial time in the size of the graph (numerical approximation schemes for the unbiased case are provided in~\cite{oberman:scheme}). We provide an algorithm, running in polynomial time in the number of vertices, computing the unique solution to the biased infinity Laplacian Boundary Problem on finite graphs, for any bias parameter. This solves the open problem suggested in~\cite[Section 7.1]{peres-p-s:biased}.

\subsection{Notation and general assumptions} Fix a probability $p \in (0,1)$, the complementary probability $q=1-p$ and the ratio $r=q/p$. Clearly, any of the numbers $p$, $q$, and $r$ uniquely determines the other two (in fact $p=1/(1+r)$ and $q=r/(1+r)$). Note that some of the discussion below is valid in the case when $p=1$ and $q=r=0$. This case is trivial enough on its own and we will not pursue it systematically.

All graphs are connected, undirected, and simple (no loops, no multiple edges). When we use $G$ to denote a graph we assume that its set of vertices is denoted by $V$ and its set of edges by $E$. Similarly, the sets of vertices and edges of $G_*$ are $V_*$ and $E_*$, respectively, and so on. By definition, no vertices in a path repeat (not even the initial and the terminal vertex). We sometimes denote a path with initial vertex $x$ and terminal vertex $y$ by $[x,y]$. Even though this notational convention may not be precise, we find it useful, especially when the intermediate vertices are not relevant in the discussion. The length of the path $[x,y]$ is denoted by $\ell[x,y]$. The graphs will usually come with a boundary, which may be any fixed and nonempty set of vertices $V_0$ of $G$. We write $x \sim y$ to indicate that $x$ and $y$ are neighbors.

\subsection{Biased infinity Laplacian}

\begin{definition}[$r$-biased infinity Laplacian and infinity harmonic functions]
Let $G$ be a connected graph with boundary $V_0$.

The $r$-biased infinity Laplacian $\Delta_\infty^r$ on the space of real functions on $G$ is given by
\[
 (\Delta_\infty^r u)(x) =
 \begin{cases}
  \displaystyle{p \cdot \max_{y \sim x} \{u(y)\} + q \cdot \min_{y \sim x} \{u(y)\} - u(x)}, & x \in V \setminus V_0 \\
  u(x), & x \in V_0
  \end{cases}.
\]

A function $u: V(G) \to \RR$ is $r$-biased infinity harmonic if
\[
 \Delta_\infty^r u  = 0,
\]
outside of the boundary $V_0$.
\end{definition}

\begin{definition}[$r$-biased infinity Laplacian Boundary Problem]
Let $G$ be a connected graph with boundary $V_0$ and $g: V_0 \to \RR$ a real function. We call $g$ the boundary condition.

A real function $u: V \to \RR$ is a solution of the $r$-biased infinity Laplacian Boundary Problem
\begin{equation}\label{e:bp}
 \begin{matrix}
  \Delta_\infty^r u &= 0 \\
  u_{|V_0} &= g
 \end{matrix}
\end{equation}
if it is an $r$-biased infinity harmonic function on $G$ that agrees with $g$ on the boundary $V_0$.
\end{definition}

The following is a basic example (it plays an essential role in our algorithm).

\begin{example}\label{eg:harmonic}
Consider the graph $G$ which is a path $x_0,\dots,x_n$ of length $n$, with boundary $V_0=\{x_0,x_n\}$ and boundary condition $g(x_0)=m$ and $g(x_n)=M$, where $m \leq M$. The unique solution to the $r$-biased infinity Laplacian Boundary Problem, for $r \neq 1$, is given by
\[
 u(x_i) = A + Br^i, \qquad i = 0,\dots, n,
\]
where
\[
 A = \frac{M-r^n m}{1-r^n}, \qquad B = \frac{m-M}{1-r^n},
\]
and for $r=1$, by
\[
 u(x_i) = m + \frac{(M-m)i}{n}, \qquad i = 0,\dots, n.
\]

Note that, regardless of whether $r\neq 1$ or $r=1$,
\[
 u(x_{i+1}) - u(x_i) = \frac{(M-m)r^i}{1+r+\dots+r^{n-1}},
\]
for $i = 0,\dots,n-1$. Thus the values of $u$ are nondecreasing as we move along the path from $x_0$ to $x_n$. In fact, they are strictly increasing, unless $M=m$,  in which case they are all equal to $m$.
\end{example}

\subsection{Biased tug-of-war on graphs}

Let $G$ be a finite graph with boundary $V_0$ and boundary condition $g:V_0 \to \RR$. We may think of $g$ as a pay-off function in a tug-of-war game on $G$ that is played as follows. A token is placed at a vertex in $V \setminus V_0$. Two players, Player I and Player II flip a coin to decide who makes the next move. The probability that Player I wins is $p$ and the probability that Player II wins is $q$. Each time a player earns the right to move, he moves the token to a neighboring vertex, as he pleases, at which point the coin is tossed again to decide who makes the next move. The game stops when the token reaches a vertex $x$ in $V_0$ at which point Player I wins the amount $g(x)$ from Player II (the game is a zero-sum game). Player I tries to maximize this pay-off and Player II tries to minimize it. The solution $u$ to the Boundary Problem~\eqref{e:bp} is the value of the game, i.e., for every vertex $x$, $u(x)$ is the expected pay-off for Player I under optimal strategy of play by both players when the game starts with the token at $x$. Moreover, the optimal strategies for both players are revealed by the value function. Whenever Player I wins the turn, he needs to move the token to a neighbor with the maximal value and whenever Player II wins the turn, he needs to move to a neighbor with the minimal value (in case of ties for maximum or minimum, one may choose any of the vertices that are tied).

\subsection{Main result}

\begin{theorem}
Let $G$ be a graph with boundary $V_0$ and boundary condition $g:V_0 \to \RR$. There exists an algorithm running in polynomial time in the number of vertices of $G$ that extends $g$ to a solution $u$ of the boundary problem~\eqref{e:bp}.
\end{theorem}

Let us briefly (and loosely) explain the approach behind the algorithm. A path in $G$ connecting two boundary points is called a connection. Let $u$ be the $r$-biased infinity harmonic function solving the boundary problem for $G$. An $r$-biased infinity harmonic connection in $G$ is a connection path such that the restriction of $u$ on the path is the solution of the boundary problem on the path (with the endpoints as boundary points). A crucial feature of the harmonic function $u$ solving the boundary problem on $G$ is that there is always a harmonic connection (this is not obvious at all, and we do not prove it directly, but it becomes apparent after the solution is constructed by the algorithm). Thus, if we know such a harmonic connection we can immediately determine the $u$ values along that path, by using the formula from Example~\ref{eg:harmonic}. The vertices along the harmonic connection that are assigned values in this way may then be considered part of the boundary and we may continue by looking for a harmonic connection path on the new graph (with larger boundary). We continue in this way until all vertices are assigned values. Thus, the boundary problem can be solved efficiently if we find an efficient way of recognizing harmonic connection from the initial data. The criterion employed to determine such a path is to choose the path of steepest ascent connecting boundary points, where the steepest ascent is interpreted in an appropriate way by defining the notion of $r$-slope along connections.

\subsection{Existence and uniqueness of solutions}

The existence of solutions follows from our algorithm, since we prove that the algorithm always stops and produces a solution. The uniqueness also follows from our results, since we prove that any solution must contain $r$-biased infinity harmonic connections, and the values along such connections are uniquelly determined. 

Our approach produces exact solutions in time polynomial in the size of the input graph. We quickly indicate another approach, based on approximations. 

Assuming $m = \min\{g(x)\mid x \in V_0\}$ and $M = \max\{g(x)\mid x \in V_0\}$, start with an initial function $u_0: V \to \RR$, defined by $u_0(x)=m$, for $x \in V \setminus V_0$ and $u_0(x)=g(x)$, for $x \in V_0$. For $n \geq 0$, define the function $u_{n+1}: V \to \RR$, iteratively, by
\[
 u_{n+1} =
  \begin{cases}
   p \cdot \sup_{y \sim x} \{u_n(y)\} + q \cdot \inf_{y \sim x} \{u_n(y)\}, & x \in V \setminus V_0 \\
   g(x), & x \in V_0
  \end{cases}.
\]
It is easy to see that $u_{n+1}(x)\geq u_n(x)$, for all $n \geq 0$ and all $x \in V$, i.e., the sequence of functions $\{u_n\}_{n=0}^\infty$ is nondecreasing in each coordinate (at each vertex), and since the values are bounded above by $M$, the sequence monotonously converges, from below, to some function $u$, which is an $r$-biased harmonic extension of $g$. Therefore solutions exist and can be approximated to any accuracy by the suggested iteration procedure. The only problem is that we do not know, at present, any good bounds on the number of steps needed to achieve some prescribed accuracy.

Note that if we start the iteration from above, by defining $u_0(x)=M$, for $x \in V \setminus V_0$ and $u_0(x)=g(x)$, for $x \in V_0$, the sequence will also monotonously converge to the solution, but this time from above. In particular, if we simultaneously run the iteration from above and from below, we obtain, at each step, upper and lower bounds for the solution, both of which converge to the solution. In this way we obtain, at each step, both an approximation to the solution and an estimate of the error.

If, at any moment, the gap between the lower and the upper bounds becomes sufficiently narrow to allow us to determine, for each vertex $x \in V \setminus V_0$, which neighbor of $x$ has the minimum and which neighbor has the maximum value, we may stop the approximations and calculate all values exactly, since, once we know the correct directions for the minima and maxima (i.e., we know the optimal strategy for both players in the game) the problem reduces to solving a linear system of equations (one equation for each vertex in $V \setminus V_0$). While this might be rather practical approach that provides the solution quickly in many situations, there are two problems with it in general, one is that we do not know how long does it take for this separation of values to occur and, in some situations, when two or more neighbors have the same maximal or minimal value, the separation may never occur at all.

Nevertheless, if there are only a handful of places where the approximations cannot sufficiently separate the values, we may still efficiently use this approach. Indeed, we may fix the correct directions for the vertices for which the separation works well and try all possibilities for the remaining few cases (solve one linear system for every possible choice of directions).

In fact, we could use the approach based on guessing of the correct directions from the very beginning, without running any approximations or any other preliminary work, and solve one linear system of equations for each choice of directions until we find one that is consistent, but this algorithm is not practical as there are exponentially many choices in general.


\section{The case of non-constant bias}

Before we present the algorithm solving the boundary problem in the case of a constant bias $r$ at all vertices we provide a few comments and examples regarding the general case in which the bias may be different at different vertices. Given probabilities $p_x \in (0,1)$, for $x \in V \setminus V_0$, that depend on the vertices, a biased harmonic function $u:V \to \RR$ is a function $u$ that satisfies
\[
p_x \cdot \max_{y \sim x} \{u(y)\} + q_x \cdot \min_{y \sim x} \{u(y)\} - u(x) = 0,
\]
for $x \in V \setminus V_0$, or, equivalently,
\[
 \max_{y \sim x} \{u(y)\} + r_x \cdot \min_{y \sim x} \{u(y)\} = (1+r_x)u(x),
\]
where $q_x=1-p_x$ and $r_x=q_x/p_x$, for $x \in V\setminus V_0$.

\begin{example}\label{eg:nopath}
In the graph in Figure~\ref{f:nonconsant} the boundary vertices are the vertices $b_0$ and $b_1$ with values 0 and 9, respectively (double square frames are used to indicate these vertices) and there are 4 vertices in $V \setminus V_0$, namely $x$, $y$, $z$, and $w$ (indicated by round frames). The bias is 1 at all vertices ($r_y=r_z=r_w=1$) except at the vertex $x$ at which the bias is $r_x=3$ (probabilities $p_x=1/4$ and $q_x=3/4$). The encircled values at every vertex define a biased harmonic function solving the boundary problem. The arrows do not indicate directions on the edges (recall that the graphs are undirected). Rather, the different line styles used on edge halves indicate the relationship between the endpoints of that edge. For every vertex $x$ in $V \setminus V_0$, there is an outgoing  single arrow ($\rightarrow$) pointing toward the neighbor with maximum value and an outgoing double arrow ($\twoheadrightarrow$) pointing toward the neighbor with minimum value. An edge connecting a vertex to a neighbors whose value is neither minimal nor maximal is indicated by a dotted line and an edge connecting a boundary vertex to a vertex in $V \setminus V_0$ by a full line. In terms of the tug-of-war interpretation each single arrow represents the optimal move for Player I (i.e., Player I should always move to the neighbor in the direction of the single arrow) and each double arrow represents the optimal move for Player II from the given vertex.
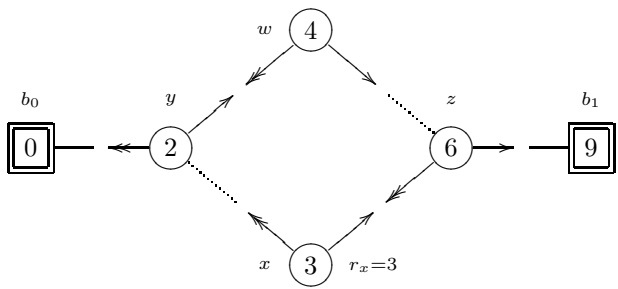
\begin{figure}[!ht]
\[
\xymatrix@R=10pt@C=15pt{
 &&&& *++[o][F]{4} \ar@{->}[dr] \ar@{->>}[dl] \ar@{}[l]|<<<<{w} &&&&
 \\
 &&&& &&&&
 \\
 *++[F=]{0} \ar@{-}[r] \ar@{}[u]|<<<<{b_0} &&
 *++[o][F]{2} \ar@{->>}[l] \ar@{->}[ur] \ar@{..}[dr] \ar@{}[u]|<<<<{y}&&&&
 *++[o][F]{6} \ar@{->}[r] \ar@{->>}[dl] \ar@{..}[ul] \ar@{}[u]|<<<<{z} &&
 *++[F=]{9} \ar@{-}[l] \ar@{}[u]|<<<<{b_1}
 \\
 &&&& &&&&
 \\
 &&&& *++[o][F]{3} \ar@{->}[ur] \ar@{->>}[ul] \ar@{}[r]|>{r_x=3} \ar@{}[l]|<<<<{x} &&&&
}
\]
\caption{A biased harmonic function (under nonconstant bias)}
\label{f:nonconsant}
\end{figure}

It is easy to see that there are no biased infinity harmonic connections between the boundary vertices $b_0$ and $b_1$ (one cannot go from $b_0$ to $b_1$ without using a partially dotted edge). This means that the approach that relies on looking for such paths, which leads to a polynomial time algorithm in the case of constant bias, cannot be used in the case of nonconstant bias.
\end{example}

In contrast with the general case, we show that harmonic connections do exist even in the case of nonconstant bias as long as the underlying graph is a tree. Thus, at least in the case of trees the approach based on finding harmonic connections paths may still work, but we do not have a quick criterion at present to determine such connections from the initial data.

Note that, in the case of trees, we may always assume that the boundary is exactly the set of leaves. Indeed, any boundary problem on a tree in which the boundary is not the set of leaves can be reduced to several separate boundary problems in which all boundary points are leaves, by detaching the tree at all interior boundary vertices. Further, the boundary problem on a tree in which not all leaves are boundary points can be reduced to a problem on a smaller tree by observing that a leave that is not a boundary point must have the same value as its only neighbor and that the value of the neighbor is not affected by the removal of the leaf. Thus, after a few reduction of the above type (detachment at interior boundary vertices and removal of non-boundary leaves) we arrive at several separate problems in each of which the boundary is precisely the set of leaves.

\begin{proposition}
Let $G$ be a tree and let the boundary $V_0$ be the set of its leaves. For any boundary condition $g$ and arbitrary, not necessarily constant, biases at the interior vertices of the tree, there exists a biased infinity harmonic connection path in $G$.
\end{proposition}

\begin{proof}
Let $u$ be the solution to the boundary problem. Let $x_0$ be any interior vertex of degree at least 3, and let its neighbors be $x_1,\dots,x_k$. Moreover, let $x_1$ and $x_2$ be vertices of minimal and maximal value, respectively, among the neighbors of $x_0$. The removal of the vertex $x_0$, along with the edges incident to it, leads to $k$ connected components $T_1,\dots,T_k$, with representatives $x_1,\dots,x_k$, each of which is a tree. The induced subgraph $G_{x_0}=T_1 \cup T_2 \cup \{x_0\}$ is a tree that has fewer vertices than the original tree and all leaves of $G_{x_0}$ are part of the original boundary $V_0$. Moreover, all vertices in $G_{x_0}$ have the same neighbors as in $G$, except for $x_0$, which now has degree 2, but its two remaining neighbors in $G_{x_0}$ have the minimum and the maximum value among all neighbors of $x_0$ in $G$. We may continue this procedure (choose a vertex $x_0'$ in $G_{x_0}$ of degree 3 or higher and remove all but two subtrees that emanate from $x_0'$, but make sure to keep the two subtrees that include two vertices with the minimum and maximum values among the neighbors of $x_0'$, and so on) until we obtain a tree in which all interior vertices have degree 2. Since the only tree in which all interior vertices have degree 2 is a path, we obtain a connection between two leaves of the tree $G$. Moreover, the obtained path is a harmonic connection, since each interior vertex $x$ on this path has two of its original neighbors next to it, and these two neighbors have the minimal and the maximal possible value among all neighbors of $x$ in $G$.
\end{proof}

\begin{remark}
We note that each instance of the biased infinity Laplacian boundary problem with nonconstant bias is equivalent to such a problem on a larger graph with only two distinct boundary values (thus, in a sense, the problem with only two distinct boundary values has the same degree of difficulty as the problem with arbitrary number of distinct boundary values).

Indeed, choose two values $m$ and $M$ that are strictly smaller and strictly greater, respectively, than the boundary values $g(x)$, for $x \in V_0$. Construct a new graph $G'$ by adding, for each boundary vertex $x$, two vertices $x_m$ and $x_M$ and two edges connecting these two vertices to $x$, declaring the new boundary to be $V_0'=\{x_m,x_M \mid x \in V_0\}$, assigning the boundary values $g'(x_m)=m$ and $g'(x_M)=M$, for $x \in V_0$, and placing the bias $r_x = \frac{M-g(x)}{g(x)-m}$ at vertex $x$, for $x \in V_0$. The solution to the boundary problem on the new graph $G'$ with boundary $V_0'$ and boundary condition $g'$, restricted to $V$, is the solution to the original problem. Indeed, since the only boundary values in $G'$ are $m$ and $M$, the optimal strategy at each vertex $x$ in the old boundary $V_0$ (note that no vertex in the old boundary is part of the new boundary) is clear. Namely, if given the right of turn, Player I moves to $x_M$ (since there is no higher possible payoff than $M$), while Player II moves to $x_m$. Therefore, the value at $x$ of the biased harmonic function $u'$ solving the boundary problem on $G'$ is
\[
 u'(x) = \frac{r_x}{r_x+1} \cdot m + \frac{1}{r_x+1} \cdot M = \frac{M-g(x)}{M-m}\cdot m + \frac{g(x)-m}{M-m} M = g(x),
\]
which implies that $u=u'|_V$, the restriction of $u'$ to $V$, solves the boundary problem on $G$.
\end{remark}


\section{Biased slope on paths in graphs}

The notion of $r$-biased slope extends the notion of slope used in the unbiased case in~\cite{lazarus-etal:algorithm}.

\begin{definition}
Let $G$ be a graph, $g:V' \to \RR$ a real function defined on a subset $V' \subseteq V$ of vertices, and $[x,y]$ a path in $G$ that starts and ends in $V'$.

The $r$-biased slope (or $r$-slope) of $g$ on $[x,y]$, denoted $\sigma_r g[x,y]$, is
\[
 \sigma_r g[x,y] = \frac{g(y)-r^{\ell[x,y]} g(x)}{1+r+\dots+r^{\ell[x,y]-1}}.
\]
\end{definition}

\begin{remark}
Note that the value of the $r$-slope on the path $[x,y]$ depends only on the values of the function on the endpoints $x$ and $y$ and the length of the path. The function $g$ does not need to be defined on the intermediate vertices.

In the special case $r=1$ (corresponding to the unbiased case $p=q=1/2$), the slope can be rewritten as
\[
 \sigma_1 g[x,y] = \frac{g(y)-g(x)}{\ell[x,y]},
\]
which agrees with the definition of slope used in~\cite{lazarus-etal:algorithm} for the unbiased case.
\end{remark}

\begin{remark}
When $r \neq 1$, the slope can be rewritten in a more compact way as
\begin{equation}\label{e:compact}
 \sigma_r g[x,y] = \frac{g(y) - r^{\ell[x,y]} g(x)}{1-r^{\ell[x,y]}}(1-r).
\end{equation}
The expression on the right hand side tends to $\sigma_1 g[x,y]$ as $r$ tends to 1.

The form~\eqref{e:compact} may be more useful when the $r$-slope is defined on paths in metric spaces in which the lengths may not be integers. For instance, for functions on $\RR$ one may define $r$-derivatives as $\displaystyle{\lim_{h \to 0}} \frac{g(x+h)-r^hg(x)}{1-r^h}(1-r)$.
\end{remark}

The following lemma, which we give without a proof (part (i) is a simple and direct calculation, and part (ii) is an easy corollary), shows that the $r$-slope on a path $[x,y]$ is a convex linear combination of the slopes on the pieces that constitute $[x,y]$. Note that, if $r>1$ the $r$-slopes on the edges near the beginning of the path are assigned larger weights (the biased slope favors the beginning of the path), and if $r<1$ the $r$-slopes on the edges near the end of the path are assigned larger weights (the biased slope favors the end of the path).

\begin{lemma}\label{l:convex}
Let $G$ be a graph and $g:V' \to \RR$ a real function defined on a subset $V' \subseteq V$ of vertices.

\textup{(i)}
If $x,y,z$ are vertices in $V'$, $[x,y]$ is a path in $G$, $z$ is a vertex on the path $[x,y]$, $\ell[x,z]=m$ and $\ell[z,y]=k$, then
\[
 \sigma_r g[x,y] = \frac{r^k+r^{k+1}+\dots+r^{k+m-1}}{1+r+\dots+r^{m+k-1}} \sigma_r g[x,z] +
                   \frac{1+r+\dots+r^{k-1}}{1+r+\dots+r^{m+k-1}} \sigma_r g[z,y].
\]

\textup{(ii)}
If $x_0,x_1,\dots,x_n$ is a path in $G$ consisting of vertices in $V'$, then
\[
 \sigma_r g[x_0,x_n] = \sum_{i=1}^n \frac{r^{n-i}}{1+r+\dots+r^{n-1}} \sigma_r g[x_{i-1},x_i].
\]
\end{lemma}

\begin{remark}\label{r:increasing}
We compare the $r$-slope on a path $[x,y]$ and its reversal $[y,x]$. Assume that $g(y) \geq g(x)$. In that case we say that the $r$-slope $\sigma_r g[x,y]$ is taken in the increasing direction and the $r$-slope $\sigma_r g[y,x]$ is taken in the decreasing direction. Note that the slope may be negative even when taken in the increasing direction, and it may be nonzero even when the $g$ values on the endpoints $x$ and $y$ are equal. This seemingly unusual behavior becomes more intuitively acceptable when one realizes that, loosely speaking, the $r$-slope ``compares'' the function $g$ to the exponential function with base $r$ along the paths and not to the constant function as the usual slope does. In particular, when $g$ is ``small`` compared to $r$ all slopes in the graph may be negative (for instance, take $g$ to be constant 1 and $r>1$)

The difference of the $r$-slopes on $[x,y]$ of length $n$ and its reversal $[y,x]$ is
\[
 \sigma_r g[x,y] - \sigma_r g[y,x] = \frac{(g(y)-g(x))(1+r^n)}{1+r+\dots+r^{n-1}},
\]
which shows that
\[
 \sigma_r g[x,y] > \sigma_r g[y,x] \qquad \text{if and only if} \qquad g(y) > g(x)
\]
and
\[
 \sigma_r g[x,y] = \sigma_r [y,x] \qquad \text{if and only if} \qquad g(y) = g(x).
\]
In other words, the $r$-slope taken in the increasing direction is greater than or equal to the $r$-slope taken in the decreasing direction (as one expects from a slope), and the equality of these two $r$-slopes is achieved if and only if the $g$ values at the two ends of the path are the same.
\end{remark}

\begin{remark}\label{r:short-steep}
The $r$-slope on a path $[x,y]$ of length $n$ can be rewritten as
\[
 \sigma_r g[x,y] = \frac{g(y)-g(x)}{1+r+\dots+r^{n-1}} +(1-r)g(x).
\]
Therefore, if
\[
 g(x)=g(x') \leq g(y) = g(y'),
\]
then
\[
 \sigma_r g[x,y] < \sigma_r g[x',y'] \qquad \text{if and only if} \qquad \ell[x',y'] < \ell[x,y].
\]
In other words, when the same increase in the value of $g$ is achieved on a shorter path, the $r$-slope is greater (as one expects from a slope).
\end{remark}

\begin{remark}\label{r:reciprocal}
The $r$-slope and the $(1/r)$-slope are related by
\[
 \sigma_r g[x,y] = - r \sigma_{1/r} g[y,x].
\]
\end{remark}

We show that, in the case of paths, the behavior of the $r$-slope on its edges characterizes the $r$-biased infinity harmonic functions.

\begin{proposition}\label{p:constant}
Let $G$ be the graph which is a path $x_0,x_1,\dots,x_n$ of length $n$, $u: V \to \RR$ be a real valued function and $u(x_0) = m \leq M = u(x_n)$. The following are equivalent.

(i) All edge $r$-slopes $\sigma_r u[x_i,x_{i+1}]$, for $i=0,\dots,n-1$, are equal.

(ii) The function $u$ is $r$-biased infinity harmonic function with respect to the boundary $V_0=\{x_0,x_n\}$.
\end{proposition}

\begin{proof}
Assume (i). We have, for $i=1,\dots,n-1$,
\begin{equation}\label{e:difference}
 u(x_{i+1}) - r u(x_i) =  u(x_{i}) - r u(x_{i-1}).
\end{equation}
After multiplication by $p$ the last equality may be rewritten as
\[
 pu(x_{i+1}) + qu(x_{i-1}) - u(x_i)= 0,
\]
which shows that $u$ is $r$-biased infinity harmonic with respect to the boundary $V_0=\{x_0,x_n\}$, provided the sequence of values $u(x_0),u(x_1),\dots,u(x_n)$ is nondecreasing. However, the equality~\eqref{e:difference} can also be rewritten as
\[
 u(x_{i+1}) - u(x_i) =  r(u(x_{i}) - u(x_{i-1})),
\]
which shows that all differences $u(x_{i+1}) - u(x_i)$ have the same sign and the sequence is either strictly increasing, strictly decreasing, or constant. Since $m \leq M$, the sequence is either strictly increasing (if $m<M$) or constant (if $m=M$). Therefore (i) implies (ii).

Assume (ii). From Example~\ref{eg:harmonic} we have, in case $r \neq 1$,
\[
 \sigma_r u[x_i,x_{i+1}] = u(x_{i+1}) - r u(x_i) = A(1-r),
\]
for $i=0,\dots,n-1$, which implies that
\begin{equation}\label{e:constant}
 \sigma_r u[x_i,x_{i+1}] = \frac{M-r^n m}{1+r+\dots+r^{n-1}} = \sigma_r u[x_0,x_n],
\end{equation}
and the $r$-slope is constant on all edges and equal to the $r$-slope on the path taken in the increasing direction. The equalities in~\eqref{e:constant} and the same conclusion are also valid when $r=1$ (indeed, $\sigma_r u[x_i,x_{i+1}] = u(x_{i+1})-u(x_i)=\frac{M-m}{n}$ in this case). Therefore (ii) implies (i).
\end{proof}

\begin{remark}
Note that if the $r$-slopes of a function $u$ on a path $G$ taken in the direction compatible with the direction from $x_0$ towards $x_n$ are constant on all edges, then either $u$ is $r$-biased infinity harmonic (when $u(x_0) \leq u(x_n)$, as in Proposition~\ref{p:constant}) or $(1/r)$-biased infinity harmonic (when $u(x_0) \geq u(x_n)$). Indeed, in the latter case, by Remark~\ref{r:reciprocal}, the $(1/r)$-slopes taken in the direction compatible with the direction from $x_n$ towards $x_0$ are constant on all edges, and Proposition~\ref{p:constant} applies.
\end{remark}


\section{The algorithm}\label{s:algorithm}

We will construct the required solution to the Boundary Problem~\eqref{e:bp} through an increasing series of compatible partial extensions of $G_0=(V_0,\emptyset,g_0)$ with $g_0=g$, namely $G_0,G_1,G_2,\dots,G_N$, where each $G_i=(V_i,E_i,g_i)$, for $i=0,\dots,N$, is a graph together with a real valued function $g_i: V_i \to \RR$, the last graph $G_N$ has $V_N=V$ and $E_N=E$, the increasing property means that $V_i \subseteq V_{i+1}$ and $E_{i} \subseteq E_{i+1}$, for $i=0,\dots,N-1$, and the compatibility condition means that, for $0 \leq i \leq j \leq N$ and all $x$ in $V_i$, $g_j(x) = g_i(x)$.

\begin{definition}
Let $G_*$ be a partial extension of $G_0$.

A connecting path for $G_*$ is a path
\[ x_0,e_1,x_1,\dots,x_{n-1},e_n,x_n \]
in $G$, such that
\begin{itemize}
\item[-] the vertices $x_0$ and $x_n$ are distinct vertices in $V_*$,

\item[-] the vertices $x_1,\dots,x_{n-1}$ are distinct vertices in $V \setminus V_*$, and

\item[-] the edges $e_1,\dots,e_n$ are in $E \setminus E_*$.
\end{itemize}
\end{definition}

Note that the difference between connections, as defined before, and connecting paths is technical, the former are defined only for $G$, while the latter for partial extensions.

We are ready to describe the algorithm. The input is the graph $G=(V,E)$, the boundary $V_0 \subseteq V$, and the boundary condition $g: V_0 \to \RR$. The output is the solution $u: V \to \RR$ to the $r$-biased infinity Laplacian Boundary Problem~\eqref{e:bp}.

\begin{description}
\item[Step 1] Set $G_*$ to be the graph $G_0$.

\item[Step 2] Determine a connecting path $x_0,e_1,x_1,\dots,e_n,x_n$ for $G_*$ with largest possible $r$-slope $\sigma_r [x_0,x_n]$.  If there are no connecting paths to be found, skip to Step 5, otherwise continue to Step 3.

\item[Step 3] Set $V_{**} = V_{*} \cup \{x_1,\dots,x_{n-1}\}$, $E_{**}=E_* \cup \{e_1,\dots,e_n\}$, and $g_{**}(x)=g_*(x)$, for $x \in V_*$. If $r \neq 1$, set $g_{**}(x_i) = A+Br^i$, for $i=1,\dots,n-1$, where $A=\frac{g_*(x_n)-r^n g_*(x_0)}{1-r^n}$ and $B = \frac{g_*(x_0)-g_*(x_n)}{1-r^n}$. Otherwise set $g_{**}(x_i) = m+\frac{g_*(x_n)-g_*(x_0)}{n}$, for $i=1,\dots,n-1$. Set $G_{**}$ to be the graph with vertex set $V_{**}$ and edge set $E_{**}$ along with the function $g_{**}$.

\item[Step 4] Set $G_*$ equal to $G_{**}$ and go back to Step 2.

\item[Step 5] If $V_* \neq V$ go to Step 6. Otherwise set $u=g_*$ and stop.

\item[Step 6] Set $V_{**}=V$, $E_{**}=E$, and for every vertex $x$ in $V_{**} \setminus V_*$, determine the vertex $x_*$ from $V_*$ that is closest to $x$ and set $g_{**}(x)=g_*(x_*)$. Set $u=g_{**}$ and stop.
\end{description}

We will justify the correctness of the algorithm through a series of claims, beginning with the following key fact.

\begin{claim}\label{c:decreasing-slopes}
The sequence of connecting paths, in the order in which they are selected by the algorithm, has nonincreasing $r$-slopes.
\end{claim}

\begin{proof}
It is sufficient to show that, after a connecting path $[x_0,x_n]=x_0,x_1,\dots,x_n$ for $G_*$ with largest possible $r$-slope, say equal to $s$, is selected in Step 2 and the graph $G_{**}$ is constructed in Step 3, there are no connecting paths for $G_{**}$ of $r$-slope strictly greater than $s$. We prove this by contradiction.

Assume that $[y_0,y_{n'}]=y_0,y_1,\dots,y_{n'}$ is a connecting path for $G_{**}$ with largest possible $r$-slope, say equal to $t$, with $t>s$. By Remark~\ref{r:increasing}, the $r$-slopes $s$ and $t$ are both taken in the increasing direction (otherwise they would not be the largest possible slopes). Therefore $g_*(x_0) \leq g_*(x_n)$ and $g_{**}(y_0) \leq g_{**}(y_{n'})$. Moreover, by the definition of $g_{**}$, Proposition~\ref{p:constant}, and Lemma~\ref{l:convex}, the $r$-slope on every edge along $[x_0,x_n]$ is $s$, and so is the $r$-slope on every subpath $[x_j,x_j]=x_i,x_{i+1},\dots,x_j$, for $0 \leq i < j \leq n$.

We consider several cases, each of which leads to a contradiction.

Assume both $y_0$ and $y_{n'}$ are in $V_*$. In that case $[y_0,y_{n'}]$ is a connecting path for $G_*$ with $r$-slope larger than $s$, a contradiction with the choice of $[x_0,x_{n}]$.

Assume both $y_0$ and $y_{n'}$ are in $V_{**} \setminus V_*$, say $y_0=x_i$ and $y_{n'}=x_j$, for some $1 \leq i,j\leq n-1$. Since the $g_{**}$ values are nondecreasing along the path $x_0,x_1,\dots,x_n$, we must have that $i < j$. The $r$-slope on $[x_i,x_j]=x_i,x_{i+1},\dots,x_j$ is $s$. By Remark~\ref{r:short-steep}, since $[y_0,y_{n'}]$ and $[x_i,x_j]$ have the same endpoints and the $r$-slope on $[y_0,y_{n'}]$ is strictly larger, the path $[y_0,y_{n'}]$ is strictly shorter than $[x_i,x_j]$. But then the path $x_0,\dots,x_i,y_1,\dots,y_{n'-1},x_j,\dots,x_n$ is a connecting path for $G_*$ that is strictly shorter than $[x_0,x_n]$ and it has the same endpoints, which implies that it has strictly larger $r$-slope than $[x_0,x_n]$, a contradiction with the choice of $[x_0,x_{n}]$.

Assume $y_0$ is in $V_*$ and $y_{n'}$ is in $V_{**} \setminus V_*$, say $y_{n'}=x_j$, for some $j=1,\dots,n-1$. Then $y_0,y_1,\dots,y_{n'-1},x_j,x_{j+1},\dots,x_n$ is a connecting path for $G_*$ whose $r$-slope, By Lemma~\ref{l:convex}, is a convex linear combination of $t$ (the $r$-slope on the part of the path up to $x_j$) and $s$ (the $r$-slope on the part of the path after $x_j$). Therefore this path is a connecting path for $G_*$ with $r$-slope strictly larger than $s$, a contradiction with the choice of $[x_0,x_{n}]$.

Finally, assume $y_0$ is in $V_{**}\setminus V_*$, say $y_{0}=x_i$, for some $i=1,\dots,n-1$, and $y_{n'}$ is in $V_*$. Then $x_0,x_1,\dots,x_i,y_{1},y,\dots,y_{n'}$ is a connecting path for $G_*$ whose $r$-slope, By Lemma~\ref{l:convex}, is a convex linear combination of $s$ (the $r$-slope on the part of the path up to $x_i$) and $t$ (the $r$-slope on the part of the path after $x_i$). Therefore this path is a connecting path for $G_*$ with $r$-slope strictly larger than $s$, a contradiction with the choice of $[x_0,x_{n}]$.
\end{proof}

\begin{claim}\label{c:harmonic}
For every graph $G_*=(V_*,E_*,g_*)$ obtained before the implementation of Step 5 of the algorithm, $g_*:V_* \to \RR$ is $r$-biased infinity harmonic function on $G_*$ with respect to the boundary condition $g: V_0 \to \RR$.
\end{claim}

\begin{proof}
The claim is, vacuously, correct for $G_0$. Assume that the claim it is correct for $G_*$. We want to show that the claim is correct for the graph $G_{**}$ that is constructed by following Step 2 and Step 3, starting from the given $G_*$. Let $[x_0,x_n]=x_0,x_1,\dots,x_n$ be the connecting path that is added to $G_*$ to obtain $G_{**}$, let the $r$-slope of this path be $s$ (recall that, by maximality of $s$, we must have $g_*(x_0) \leq g_*(x_n)$).

The $r$-biased infinity Laplacian condition is fulfilled for $g_{**}$ at all vertices other than $x_0,x_1,\dots,x_{n-1},x_n$, since such vertices do not have any new neighbors they did not have in $G_*$ and the condition is, by assumption, satisfied for $g_*$. The condition is also satisfied at the vertices $x_1,\dots,x_{n-1}$, since they all have degree 2 in $G_{**}$ and the $g_{**}$ values on these vertices are constructed exactly as in Example~\ref{eg:harmonic}.

Consider the vertex $x_0$. If $x_0$ is in $V_0$ there is nothing to check. If $x_0$ is not in $V_0$ it was added at some point during the algorithm as an interior vertex in some connecting path, on which by Claim~\ref{c:decreasing-slopes}, the $r$-slope is $t \geq s$. Say $x_0',x_0,x_0''$ is the piece of that path with $g_*(x_0') \leq g_*(x_0) \leq g_*(x_0'')$. Since $g_*(x_0') \leq g_*(x_0) \leq g_{**}(x_1)$, the addition of the vertex $x_1$ does not change the minimum value on the neighbors of $x_0$. Since $t = g_*(x_0'') - r g_*(x_0)$, $s = g_{**}(x_1)- r g_{*}(x_0)$, and $s \leq t$, we have
\[
 g_{**}(x_1)- r g_{*}(x_0) \leq g_*(x_0'') - r g_*(x_0),
\]
which implies that $ g_{**}(x_1) \leq g_*(x_0'')$ and the addition of $x_1$ does not change the maximum value on the neighbors of $x_0$ either.

Similarly, if $x_n$ is in $V_0$ there is nothing to check. If $x_n$ is not in $V_0$ it was added at some point during the algorithm as an interior vertex in some connecting path, on which by Claim~\ref{c:decreasing-slopes}, the $r$-slope is $t \geq s$. Say $x_n',x_n,x_n''$ is the piece of that path with $g_*(x_n') \leq g_*(x_n) \leq g_*(x_n'')$. Since $g_{**}(x_{n-1}) \leq g_*(x_n) \leq g_{*}(x_n'')$, the addition of the vertex $x_{n-1}$ does not change the maximum value on the neighbors of $x_n$. Since $t = g_*(x_n) - r g_*(x_n')$, $s = g_{*}(x_n)- r g_{**}(x_{n-1})$, and $s \leq t$, we have
\[
 g_{*}(x_n)- r g_{**}(x_{n-1}) \leq g_*(x_n) - r g_*(x_n') ,
\]
which implies that $g_{*}(x_n') \leq g_{**}(x_{n-1})$ and the addition of $x_{n-1}$ does not change the minimum value on the neighbors of $x_n$ either.

The claim follows.
\end{proof}

\begin{claim}
The output function $u$ given by the algorithm is $r$-biased infinity harmonic function with respect to the boundary condition $g: V_0 \to \RR$.
\end{claim}

\begin{proof}
Let $G_*=(V_*,E_*,g_*)$ be the last partial extension obtained before Step 5 is implemented, which is the extension at the moment when there are no more connecting paths to be found. By Claim~\ref{c:harmonic}, $g_*:V_* \to \RR$ is $r$-biased infinity harmonic function with respect to the boundary condition $g$.

If $V_*=V$, then it must also be true that $E_*=E$, since, otherwise, any edge that is not used in $G_*$ is a connecting path for $G_*$ (of length 1). Thus, in this case, after Step 5 is executed and the algorithm stops, $u=g_*$ is indeed a solution to the Boundary Problem~\eqref{e:bp}.

Assume that $V_* \neq V$. Let $x$ be a vertex in $V \setminus V_*$ and let $x_*$ be the closest vertex to $x$ in $V_*$ (since $G$ is connected there must be a path from $x$ to some vertex in $V_*$). The vertex $x_*$ has the property that every path from $x$ to a vertex in $V_*$ goes through $x_*$ (in particular $x_*$ is the unique vertex on $V_*$ that minimizes the distance to $x$). Indeed, if there is a path from $x$ to another vertex $x_*'$ in $V_*$ that does not go through $x_*$, then there is a connecting path between $x_*$ and $x_*'$, a contradiction, since there are no connecting paths for $G_*$.

Given $x_*$ in $V_*$, let $B(x_*)$ be the set, possibly empty, of vertices $y$ in $V \setminus V_*$ such that all paths from $y$ to a vertex in $V_*$ go through $x_*$. The function $g_{**}$ assigns the same value $g_*(x_*)$ to all vertices in $B(x_*)$. For $y$ in $B(x_*)$, all neighbors of $y$ are in $B(x_*)\cup\{x_*\}$. Thus they all have the same $g_{**}$ value and the $r$-biased infinity Laplacian condition is satisfied at $y$. If $x_*$ is not in $V_0$ the condition for $g_{**}$ is satisfied at $x_*$. Indeed, since it is satisfied for $g_*$, the vertex $x_*$ has neighbors $x_*'$ and $x_*''$ in $V_*$ such that $g_*(x_*') \leq g_*(x_*) \leq g_*(x_*'')$ and $p g_*(x_*'') + q g_*(x_*') = g(x_*)$. The $g_{**}$ values of all new neighbors of $x_*$ in $B(x_*)$ are equal to $g_*(x_*)$, so the maximum and the minimum on the neighbors of $x_*$ are not changed and the condition is still satisfied for $g_{**}$ at $x_*$. If $x_*$ is in $V_0$ then there is nothing to check at that vertex. Finally, since $G$ is connected, $V_* \cup \left(\cup_{x_* \in V_*} B(x_*)\right) = V$ and, after Step 6 is executed and the algorithm stops, $u=g_{**}$ is indeed a solution to the Boundary Problem~\eqref{e:bp}.
\end{proof}

\begin{claim}
The algorithm always stops and the number of steps is bounded by a polynomial function $P(n)$ in the number of input vertices $n$.
\end{claim}

\begin{proof}
Globally speaking, at the level of the basic steps of the algorithm, there is only one loop, namely the algorithm loops through Step 2, Step 3, Step 4 and back to Step 2, until there are no more connecting paths to be found. The fact that the algorithm stops follows from the fact that each time this loop is executed at least one new edge is added. Thus, after $O(n^2)$ executions of this loop there are no more connecting paths, the algorithm goes through Step 5 or Step 6 and stops.

Of the individual steps, Step 2 can be executed in polynomial time in $n$, since it boils down to finding shortest connecting paths between all pairs of vertices in $V_*$ (recall that shorter connecting paths between same pairs of vertices have larger $r$-slope) and Step 6 can be executed in polynomial time in $n$ since it boils down to finding shortest paths from each vertex in $V \setminus V_*$ to a vertex in $V_*$.
\end{proof}


\section{Examples}

In Figure~\ref{f:small}, Figure~\ref{f:medium}, and Figure~\ref{f:large} we provide examples of $r$-biased infinity harmonic functions on the same graph $G$ with the same boundary condition, for different values of $r$.

\begin{figure}[!ht]
\[
\xymatrix@!0{
 &&&&&&& *++[o][F]{} \ar@{->}[dr]\ar@{->>}[dl] \ar@{}[llllll]|{\frac{1+4r+8r^2+8r^3+4r^4}{1+4r+8r^2+11r^3+11r^4+8r^5+4r^6+r^7}}&&&&&&&
 \\
 &&&&&&& &&&&&&&
 \\
 &&&&& *++[o][F]{} \ar@{->}[rr] \ar@{->>}[dl] \ar@{..}[ur] \ar@{}[lllll]|{\frac{1+3r+4r^2+2r^3}{1+3r+5r^2+6r^3+5r^4+3r^5+r^6}}
 &&&& *++[o][F]{}  \ar@{->}[dr] \ar@{->>}[ll] \ar@{..}[ul] \ar@{}[rrrrr]|{\frac{1+3r+5r^2+4r^3+2r^4}{1+3r+5r^2+6r^3+5r^4+3r^5+r^6}}&&&&&
 \\
 &&&&&&& &&&&&&&
 \\
 &&& *++[o][F]{.} \ar@{->}[ur] \ar@{->>}[dl] \ar@{..}[dr] \ar@{}[rrrrr]|{\frac{1+2r+2r^2}{1+3r+5r^2+6r^3+5r^4+3r^5+r^6}} &&&&&&
 && *++[o][F]{,}  \ar@{->}[dr] \ar@{->>}[dl] \ar@{..}[ul]  \ar@{}[rrr]|{\frac{1+2r+2r^2}{1+2r+2r^2+r^3}}&&&
 \\
 & &&&&&& &&&&&& &
 \\
 & *++[F=]{0} \ar@{-}[ur] \ar@{-}[rr]
 &&&& *++[o][F]{} \ar@{->}[rr] \ar@{->>}[ll] \ar@{..}[ul] \ar@{}[urr]|{\frac{1}{1+r+r^2}} &&&
 &*++[o][F]{} \ar@{->}[rr] \ar@{->>}[ll] \ar@{..}[ur] \ar@{}[ull]|{\frac{1+r}{1+r+r^2}}
 &&&& *++[F=]{1} \ar@{-}[ul] \ar@{-}[ll] &
}
\]
\caption{$r$-biased infinity harmonic function, for $r \leq 1/r_0$}
\label{f:small}
\end{figure}
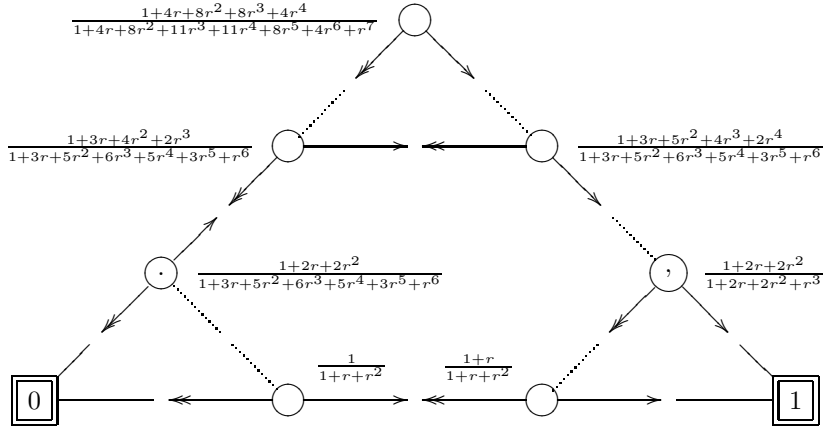

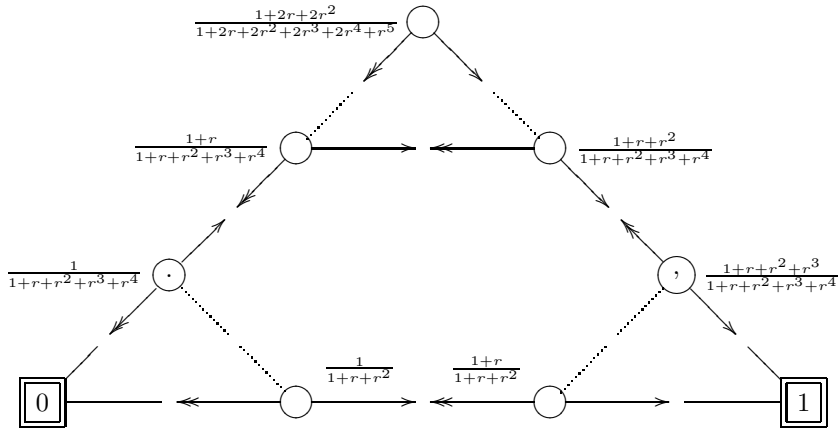
\begin{figure}[!ht]
\[
\xymatrix@!0{
 &&&&&&& *++[o][F]{} \ar@{->}[dr]\ar@{->>}[dl] \ar@{}[llll]|{\frac{1+2r+2r^2}{1+2r+2r^2+2r^3+2r^4+r^5}} &&&&&&&
 \\
 &&&&&&& &&&&&&&
 \\
 &&&&& *++[o][F]{} \ar@{->}[rr] \ar@{->>}[dl] \ar@{..}[ur] \ar@{}[lll]|{\frac{1+r}{1+r+r^2+r^3+r^4}}
 &&&&  *++[o][F]{} \ar@{->}[dr] \ar@{->>}[ll] \ar@{..}[ul] \ar@{}[rrr]|{\frac{1+r+r^2}{1+r+r^2+r^3+r^4}} &&&&&
 \\
 &&&&&&& &&&&&&&
 \\
 &&& *++[o][F]{.} \ar@{->}[ur] \ar@{->>}[dl] \ar@{..}[dr] \ar@{}[lll]|{\frac{1}{1+r+r^2+r^3+r^4}} &&&&&&
 && *++[o][F]{,}  \ar@{->}[dr] \ar@{->>}[ul] \ar@{..}[dl] \ar@{}[rrr]|{\frac{1+r+r^2+r^3}{1+r+r^2+r^3+r^4}} &&&
 \\ & &&&&&& &&&&&& &\\
 &*++[F=]{0} \ar@{-}[ur] \ar@{-}[rr]
 &&&& *++[o][F]{} \ar@{->}[rr] \ar@{->>}[ll] \ar@{..}[ul] \ar@{}[urr]|{\frac{1}{1+r+r^2}} &&&
 &*++[o][F]{}     \ar@{->}[rr] \ar@{->>}[ll] \ar@{..}[ur] \ar@{}[ull]|{\frac{1+r}{1+r+r^2}}
 &&&& *++[F=]{1} \ar@{-}[ul] \ar@{-}[ll] &
}
\]
\caption{$r$-biased infinity harmonic function, for $1/r_0 \leq r \leq r_0$}
\label{f:medium}
\end{figure}

\begin{figure}[!ht]
\[
\xymatrix@!0{
 &&&&&&& *++[o][F]{} \ar@{->}[dr]\ar@{->>}[dl] \ar@{}[llllll]|{\frac{1+4r+8r^2+7r^3+3r^4}{1+4r+8r^2+11r^3+11r^4+8r^5+4r^6+r^7}}&&&&&&&
 \\
 &&&&&&& &&&&&&&
 \\
 &&&&& *++[o][F]{} \ar@{->}[rr] \ar@{->>}[dl] \ar@{..}[ur] \ar@{}[lllll]|{\frac{1+3r+3r^2+2r^3}{1+3r+5r^2+6r^3+5r^4+3r^5+r^6}}
 &&&& *++[o][F]{}  \ar@{->}[dr] \ar@{->>}[ll] \ar@{..}[ul]  \ar@{}[rrrrr]|{\frac{1+3r+5r^2+4r^3+r^4}{1+3r+5r^2+6r^3+5r^4+3r^5+r^6}}&&&&&
 \\
 &&&&&&& &&&&&&&
 \\
 &&& *++[o][F]{.} \ar@{->}[dr] \ar@{->>}[dl] \ar@{..}[ur] \ar@{}[lll]|{\frac{1}{1+2r+2r^2+r^3}} &&&&&&
 && *++[o][F]{,}  \ar@{->}[dr] \ar@{->>}[ul] \ar@{..}[dl] \ar@{}[lllll]|{\frac{1+3r+5r^2+6r^3+3r^4+r^5}{1+3r+5r^2+6r^3+5r^4+3r^5+r^6}}&&&
 \\ & &&&&&& &&&&&& &\\
 &*++[F=]{0} \ar@{-}[ur] \ar@{-}[rr]
 &&&& *++[o][F]{} \ar@{->}[rr] \ar@{->>}[ll] \ar@{..}[ul] \ar@{}[urr]|{\frac{1}{1+r+r^2}} &&&
 &*++[o][F]{}     \ar@{->}[rr] \ar@{->>}[ll] \ar@{..}[ur] \ar@{}[ull]|{\frac{1+r}{1+r+r^2}}
 &&&& *++[F=]{1} \ar@{-}[ul] \ar@{-}[ll] &
}
\]
\caption{$r$-biased infinity harmonic function, for $r_0 \leq r$}
\label{f:large}
\end{figure}
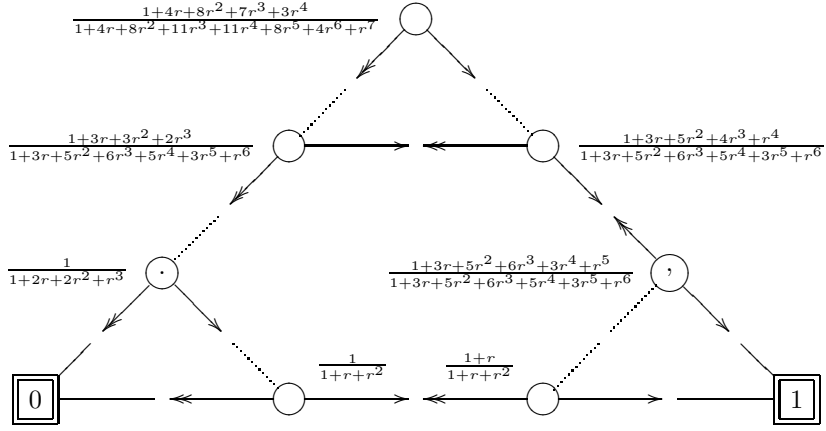

The boundary vertices are the two corner vertices in the bottom row with values 0 and 1 (double square frames are used to indicate these vertices). We keep the conventions established in Example~\ref{eg:nopath}. Thus, for every vertex $x$ in $V \setminus V_0$, there is an outgoing  single arrow ($\rightarrow$) pointing toward the neighbor with maximum value and an outgoing double arrow ($\twoheadrightarrow$) pointing toward the neighbor with minimum value. Since the two boundary vertices have values 0 and 1, the value at each vertex may be interpreted as the winning probability for Player I, under optimal play of both players, when the token is at that vertex (Player I wins if the token reaches the lower right corner and loses if the token reaches the bottom left corner).

Depending on $r$, there are exactly three different combined optimal strategies of Player I and Player II, as indicated in the three figures. The critical values $r_0 \approx 1.3247$ and $1/r_0 \approx 0.7549$ at which the behavior changes are the unique real roots of the polynomials $z^3-z-1$ and $z^3+z^2-1$, respectively.

For instance, if the token is in the third row on the left (the vertex indicated by a dot in the circle) and $r \leq r_0$, Player I should ``play conservatively'' and move the token up and to the right, away from the winning corner for Player II, and if $r \geq r_0$, Player I should ``go for broke'' and move down and to the right, ignoring the proximity of the winning corner for Player II. Note that the latter strategy is optimal for Player I only when the bias of the game is moderately to extremely favorable to Player II (approximately when $p \leq 0.43$). Symmetrically, when the token is in the third row on the right (the vertex indicated by a comma in the circle), the optimal play of Player II depends on whether $r \leq 1/r_0$ or $r \geq 1/r_0$ (the ``go for broke'' strategy of going down and to the left occurs when $r \leq 1/r_0$, which is approximately when $p \geq 0.57$). In the medium range, $1/r_0 \leq r \leq r_0$, both players should ``play conservatively''.


\newcommand{\etalchar}[1]{$^{#1}$}


\end{document}